\documentclass{lmcs}
\pdfoutput=1
\usepackage[utf8]{inputenc}

\usepackage{lastpage}
\lmcsdoi{22}{4}{3}
\lmcsheading{}{\pageref{LastPage}}{}{}%
{Feb.~17,~2026}{Oct.~08,~2026}{}

\usepackage{hyperref}

\usepackage{amsmath}
\usepackage{amssymb}
\usepackage{amsthm}

\usepackage{tikz}
\usetikzlibrary{fadings}

\def\TDINC{\mathsf{TDINC}}
\def\TDDEC{\mathsf{TDDEC}}

\def\NON{\mathsf{NON}}

\def\deg{\mathrm{deg}}
\def\r{\mathrm{r}}

\makeatletter
\newcommand{\douwidehat}[2]{%
  \sbox0{$\m@th#1\widehat{\hphantom{#2}}$}%
  \sbox2{$\m@th#1x$}
  \sbox4{$\m@th#1#2$}
  \dimen0=\ht0
  \advance\dimen0 -.8\ht2
  \dimen2=\dp4
  \rlap{%
    \raisebox{\dimexpr\dimen0-\dimen2}{%
      \scalebox{1}[-1]{\box0}%
    }%
  }%
  {#2}%
}
\makeatother

\title{Loops, Inverse Limits and Non-Determinism}

\author[V.~Brattka]{Vasco Brattka\lmcsorcid{0000-0003-4664-2183}}
\address{Faculty of Computer Science, Universit\"at der Bundeswehr M\"unchen, Germany and
Department of Mathematics and Applied Mathematics, University of Cape Town, South Africa}
\email{Vasco.Brattka@cca-net.de}

\begin{document} 

% Abbreviations

% Calligraphical font

\def\AA{{\mathcal A}}
\def\BB{{\mathcal B}}
\def\CC{{\mathcal C}}
\def\DD{{\mathcal D}}
\def\EE{{\mathcal E}}
\def\FF{{\mathcal F}}
\def\GG{{\mathcal G}}
\def\HH{{\mathcal H}}
\def\II{{\mathcal I}}
\def\JJ{{\mathcal J}}
\def\KK{{\mathcal K}}
\def\LL{{\mathcal L}}
\def\MM{{\mathcal M}}
\def\NN{{\mathcal N}}
\def\OO{{\mathcal O}}
\def\PP{{\mathcal P}}
\def\QQ{{\mathcal Q}}
\def\RR{{\mathcal R}}
\def\SS{{\mathcal S}}
\def\TT{{\mathcal T}}
\def\UU{{\mathcal U}}
\def\VV{{\mathcal V}}
\def\WW{{\mathcal W}}
\def\XX{{\mathcal X}}
\def\YY{{\mathcal Y}}
\def\ZZ{{\mathcal Z}}

% Bold font

\def\bA{{\mathbf A}}
\def\bB{{\mathbf B}}
\def\bC{{\mathbf C}}
\def\bD{{\mathbf D}}
\def\bE{{\mathbf E}}
\def\bF{{\mathbf F}}
\def\bG{{\mathbf G}}
\def\bH{{\mathbf H}}
\def\bI{{\mathbf I}}
\def\bJ{{\mathbf J}}
\def\bK{{\mathbf K}}
\def\bL{{\mathbf L}}
\def\bM{{\mathbf M}}
\def\bN{{\mathbf N}}
\def\bO{{\mathbf O}}
\def\bP{{\mathbf P}}
\def\bQ{{\mathbf Q}}
\def\bR{{\mathbf R}}
\def\bS{{\mathbf S}}
\def\bT{{\mathbf T}}
\def\bU{{\mathbf U}}
\def\bV{{\mathbf V}}
\def\bW{{\mathbf W}}
\def\bX{{\mathbf X}}
\def\bY{{\mathbf Y}}
\def\bZ{{\mathbf Z}}

% Blackboard font

\def\IB{{\mathbb{B}}}
\def\IC{{\mathbb{C}}}
\def\IF{{\mathbb{F}}}
\def\IN{{\mathbb{N}}}
\def\IP{{\mathbb{P}}}
\def\IQ{{\mathbb{Q}}}
\def\IR{{\mathbb{R}}}
\def\IS{{\mathbb{S}}}
\def\IT{{\mathbb{T}}}
\def\IZ{{\mathbb{Z}}}

\def\IIB{{\mathbb{\mathbf B}}}
\def\IIC{{\mathbb{\mathbf C}}}
\def\IIN{{\mathbb{\mathbf N}}}
\def\IIQ{{\mathbb{\mathbf Q}}}
\def\IIR{{\mathbb{\mathbf R}}}
\def\IIZ{{\mathbb{\mathbf Z}}}

% Definitions

%\def\IF{\quad\mbox{if}\quad}
\def\ELSE{\quad\mbox{else}\quad}
\def\WITH{\quad\mbox{with}\quad}
\def\FOR{\quad\mbox{for}\quad}
\def\AND{\;\mbox{and}\;}
\def\OR{\;\mbox{or}\;}

\def\To{\longrightarrow}
\def\TO{\Longrightarrow}
\def\In{\subseteq}
\def\sm{\setminus}
\def\Inneq{\In_{\!\!\!\!/}}
\def\dmin{\mathop{\dot{-}}}
\def\splus{\oplus}
\def\SEQ{\triangle}
\def\DIV{\uparrow}
\def\INV{\leftrightarrow}
\def\SET{\Diamond}

\def\kto{\equiv\!\equiv\!>}
\def\kin{\subset\!\subset}
\def\pto{\leadsto}
\def\into{\hookrightarrow}
\def\onto{\to\!\!\!\!\!\to}
\def\prefix{\sqsubseteq}
\def\rel{\leftrightarrow}
\def\mto{\rightrightarrows}

\def\B{{\mathsf{{B}}}}
\def\D{{\mathsf{{D}}}}
\def\G{{\mathsf{{G}}}}
\def\E{{\mathsf{{E}}}}
\def\J{{\mathsf{{J}}}}
\def\K{{\mathsf{{K}}}}
\def\L{{\mathsf{{L}}}}
\def\R{{\mathsf{{R}}}}
\def\T{{\mathsf{{T}}}}
\def\U{{\mathsf{{U}}}}
\def\W{{\mathsf{{W}}}}
\def\Z{{\mathsf{{Z}}}}
\def\w{{\mathsf{{w}}}}
\def\HP{{\mathsf{{H}}}}
\def\C{{\mathsf{{C}}}}
\def\Tot{{\mathsf{{Tot}}}}
\def\Fin{{\mathsf{{Fin}}}}
\def\Cof{{\mathsf{{Cof}}}}
\def\Cor{{\mathsf{{Cor}}}}
\def\Equ{{\mathsf{{Equ}}}}
\def\Com{{\mathsf{{Com}}}}
\def\Inf{{\mathsf{{Inf}}}}

\def\Tr{{\mathrm{Tr}}}
\def\Sierp{{\mathrm Sierpi{\'n}ski}}
\def\psisierp{{\psi^{\mbox{\scriptsize\Sierp}}}}
\def\cl{{\mathrm{{cl}}}}
\def\Haus{{\mathrm{{H}}}}
\def\Ls{{\mathrm{{Ls}}}}
\def\Li{{\mathrm{{Li}}}}

\def\CL{\mathsf{CL}}
\def\ACC{\mathsf{ACC}}
\def\DNC{\mathsf{DNC}}
\def\ATR{\mathsf{ATR}}
\def\LPO{\mathsf{LPO}}
\def\LLPO{\mathsf{LLPO}}
\def\WKL{\mathsf{WKL}}
\def\RCA{\mathsf{RCA}}
\def\ACA{\mathsf{ACA}}
\def\SEP{\mathsf{SEP}}
\def\BCT{\mathsf{BCT}}
\def\IVT{\mathsf{IVT}}
\def\IMT{\mathsf{IMT}}
\def\OMT{\mathsf{OMT}}
\def\CGT{\mathsf{CGT}}
\def\UBT{\mathsf{UBT}}
\def\BWT{\mathsf{BWT}}
\def\HBT{\mathsf{HBT}}
\def\BFT{\mathsf{BFT}}
\def\FPT{\mathsf{FPT}}
\def\WAT{\mathsf{WAT}}
\def\LIN{\mathsf{LIN}}
\def\B{\mathsf{B}}
\def\BF{\mathsf{B_\mathsf{F}}}
\def\BI{\mathsf{B_\mathsf{I}}}
\def\C{\mathsf{C}}
\def\CF{\mathsf{C_\mathsf{F}}}
\def\CN{\mathsf{C_{\IN}}}
\def\CI{\mathsf{C_\mathsf{I}}}
\def\CK{\mathsf{C_\mathsf{K}}}
\def\CA{\mathsf{C_\mathsf{A}}}
\def\WPO{\mathsf{WPO}}
\def\WLPO{\mathsf{WLPO}}
\def\MP{\mathsf{MP}}
\def\BD{\mathsf{BD}}
\def\Fix{\mathsf{Fix}}
\def\Mod{\mathsf{Mod}}

\def\s{\mathrm{s}}
\def\r{\mathrm{r}}
\def\w{\mathsf{w}}

\def\leqm{\mathop{\leq_{\mathrm{m}}}}
\def\equivm{\mathop{\equiv_{\mathrm{m}}}}
\def\leqT{\mathop{\leq_{\mathrm{T}}}}
\def\lT{\mathop{<_{\mathrm{T}}}}
\def\nleqT{\mathop{\not\leq_{\mathrm{T}}}}
\def\equivT{\mathop{\equiv_{\mathrm{T}}}}
\def\nequivT{\mathop{\not\equiv_{\mathrm{T}}}}
\def\leqwtt{\mathop{\leq_{\mathrm{wtt}}}}
\def\equiPT{\mathop{\equiv_{\P\mathrm{T}}}}
\def\leqW{\mathop{\leq_{\mathrm{W}}}}
\def\equivW{\mathop{\equiv_{\mathrm{W}}}}
\def\leqtW{\mathop{\leq_{\mathrm{tW}}}}
\def\leqSW{\mathop{\leq_{\mathrm{sW}}}}
\def\equivSW{\mathop{\equiv_{\mathrm{sW}}}}
\def\leqPW{\mathop{\leq_{\widehat{\mathrm{W}}}}}
\def\equivPW{\mathop{\equiv_{\widehat{\mathrm{W}}}}}
\def\leqFPW{\mathop{\leq_{\mathrm{W}^*}}}
\def\equivFPW{\mathop{\equiv_{\mathrm{W}^*}}}
\def\leqWW{\mathop{\leq_{\overline{\mathrm{W}}}}}
\def\nleqW{\mathop{\not\leq_{\mathrm{W}}}}
\def\nleqSW{\mathop{\not\leq_{\mathrm{sW}}}}
\def\lW{\mathop{<_{\mathrm{W}}}}
\def\lSW{\mathop{<_{\mathrm{sW}}}}
\def\nW{\mathop{|_{\mathrm{W}}}}
\def\nSW{\mathop{|_{\mathrm{sW}}}}
\def\leqt{\mathop{\leq_{\mathrm{t}}}}
\def\equivt{\mathop{\equiv_{\mathrm{t}}}}
\def\leqtop{\mathop{\leq_\mathrm{t}}}
\def\equivtop{\mathop{\equiv_\mathrm{t}}}

\def\bigtimes{\mathop{\mathsf{X}}}

\def\leqm{\mathop{\leq_{\mathrm{m}}}}
\def\equivm{\mathop{\equiv_{\mathrm{m}}}}
\def\leqT{\mathop{\leq_{\mathrm{T}}}}
\def\leqM{\mathop{\leq_{\mathrm{M}}}}
\def\equivT{\mathop{\equiv_{\mathrm{T}}}}
\def\equiPT{\mathop{\equiv_{\P\mathrm{T}}}}
\def\leqW{\mathop{\leq_{\mathrm{W}}}}
\def\equivW{\mathop{\equiv_{\mathrm{W}}}}
\def\nequivW{\mathop{\not\equiv_{\mathrm{W}}}}
\def\leqSW{\mathop{\leq_{\mathrm{sW}}}}
\def\equivSW{\mathop{\equiv_{\mathrm{sW}}}}
\def\leqPW{\mathop{\leq_{\widehat{\mathrm{W}}}}}
\def\equivPW{\mathop{\equiv_{\widehat{\mathrm{W}}}}}
\def\nleqW{\mathop{\not\leq_{\mathrm{W}}}}
\def\nleqSW{\mathop{\not\leq_{\mathrm{sW}}}}
\def\lW{\mathop{<_{\mathrm{W}}}}
\def\lSW{\mathop{<_{\mathrm{sW}}}}
\def\nW{\mathop{|_{\mathrm{W}}}}
\def\nSW{\mathop{|_{\mathrm{sW}}}}

\def\botW{\mathbf{0}}
\def\midW{\mathbf{1}}
\def\topW{\mathbf{\infty}}

\def\pol{{\leq_{\mathrm{pol}}}}
\def\rem{{\mathop{\mathrm{rm}}}}

\def\cc{{\mathrm{c}}}
\def\d{{\,\mathrm{d}}}
\def\e{{\mathrm{e}}}
\def\ii{{\mathrm{i}}}

\def\Cf{C\!f}
\def\id{{\mathrm{id}}}
\def\pr{{\mathrm{pr}}}
\def\inj{{\mathrm{inj}}}
\def\cf{{\mathrm{cf}}}
\def\dom{{\mathrm{dom}}}
\def\range{{\mathrm{range}}}
\def\graph{{\mathrm{graph}}}
\def\Graph{{\mathrm{Graph}}}
\def\epi{{\mathrm{epi}}}
\def\hypo{{\mathrm{hypo}}}
\def\Lim{{\mathrm{Lim}}}
\def\diam{{\mathrm{diam}}}
\def\dist{{\mathrm{dist}}}
\def\supp{{\mathrm{supp}}}
\def\union{{\mathrm{union}}}
\def\fiber{{\mathrm{fiber}}}
\def\ev{{\mathrm{ev}}}
\def\mod{{\mathrm{mod}}}
\def\sat{{\mathrm{sat}}}
\def\seq{{\mathrm{seq}}}
\def\lev{{\mathrm{lev}}}
\def\mind{{\mathrm{mind}}}
\def\arccot{{\mathrm{arccot}}}
\def\cl{{\mathrm{cl}}}

\def\Add{{\mathrm{Add}}}
\def\Mul{{\mathrm{Mul}}}
\def\SMul{{\mathrm{SMul}}}
\def\Neg{{\mathrm{Neg}}}
\def\Inv{{\mathrm{Inv}}}
\def\Ord{{\mathrm{Ord}}}
\def\Sqrt{{\mathrm{Sqrt}}}
\def\Re{{\mathrm{Re}}}
\def\Im{{\mathrm{Im}}}
\def\Sup{{\mathrm{Sup}}}

\def\LSC{{\mathcal LSC}}
\def\USC{{\mathcal USC}}

\def\CE{{\mathcal{E}}}
\def\Pref{{\mathrm{Pref}}}

\def\Baire{\IN^\IN}

\def\TRUE{{\mathrm{TRUE}}}
\def\FALSE{{\mathrm{FALSE}}}

\def\co{{\mathrm{co}}}

\def\BBB{{\tt B}}

\newcommand{\SO}[1]{{{\mathbf\Sigma}^0_{#1}}}
\newcommand{\SI}[1]{{{\mathbf\Sigma}^1_{#1}}}
\newcommand{\PO}[1]{{{\mathbf\Pi}^0_{#1}}}
\newcommand{\PI}[1]{{{\mathbf\Pi}^1_{#1}}}
\newcommand{\DO}[1]{{{\mathbf\Delta}^0_{#1}}}
\newcommand{\DI}[1]{{{\mathbf\Delta}^1_{#1}}}
\newcommand{\sO}[1]{{\Sigma^0_{#1}}}
\newcommand{\sI}[1]{{\Sigma^1_{#1}}}
\newcommand{\pO}[1]{{\Pi^0_{#1}}}
\newcommand{\pI}[1]{{\Pi^1_{#1}}}
\newcommand{\dO}[1]{{{\Delta}^0_{#1}}}
\newcommand{\dI}[1]{{{\Delta}^1_{#1}}}
\newcommand{\sP}[1]{{\Sigma^\P_{#1}}}
\newcommand{\pP}[1]{{\Pi^\P_{#1}}}
\newcommand{\dP}[1]{{{\Delta}^\P_{#1}}}
\newcommand{\sE}[1]{{\Sigma^{-1}_{#1}}}
\newcommand{\pE}[1]{{\Pi^{-1}_{#1}}}
\newcommand{\dE}[1]{{\Delta^{-1}_{#1}}}

\newcommand{\dBar}[1]{{\overline{\overline{#1}}}}

\def\rand#1{\marginpar{\rule[-#1 mm]{1mm}{#1mm}}}

\def\BL{\BB}

% Commands

\newcommand{\bra}[1]{\langle#1|}
\newcommand{\ket}[1]{|#1\rangle}
\newcommand{\braket}[2]{\langle#1|#2\rangle}

\newcommand{\ind}[1]{{\em #1}\index{#1}}
\newcommand{\mathbox}[1]{\[\fbox{\rule[-4mm]{0cm}{1cm}$\quad#1$\quad}\]}

% Environments

%\newenvironment{proof}{{\mathbf Proof.}\begin{small}}{\end{small}\QED\\}
%\newenvironment{cases}{\left\{\begin{array}{ll}}{\end{array}\right.}
\newenvironment{eqcase}{\left\{\begin{array}{lcl}}{\end{array}\right.}

\newtheorem{question}[thm]{Question}

\subjclass{[{\bf Theory of computation}]:  Logic; [{\bf Mathematics of computing}]: Continuous mathematics.}

\begin{abstract}
We introduce an operator on problems in Weihrauch complexity, which we call the {\em infinite loop} or {\em inverse limit}, and which
 corresponds to an infinite compositional product.
This operation arises naturally whenever one implements algorithms that produce a sequence of results 
in an infinite loop, using some fixed subroutine.
We prove that the corresponding operator is monotone with respect to (strong) Weihrauch reducibility but that
it is not a closure operator. 
One of our findings is that weak K\H{o}nig's lemma is closed under infinite loops, which implies that the
class of non-deterministically computable problems is also closed under this operation.
Consequently, this class allows for a high degree of flexibility in programming.
As our main technical tools, we present an injective version of the recursion theorem
and an infinitary version of the so-called independent choice theorem.
We also show that, in general, the infinite loop operator is more powerful than the composition of 
 the diamond operator followed by the parallelization operator. However,  in many practical scenarios,
these compositions yield a result, which coincides with the application of the infinite loop operator.
Finally, we discuss the special situation of loops for single-valued problems and for problems on Turing degrees.
\end{abstract}

\vspace*{19pt}
\maketitle

\section{Introduction}

In this article, we explore the power of loops in Weihrauch complexity~\cite{BGP21}.
In Weihrauch complexity, a multivalued function $f:\In X\mto Y$ is seen as a {\em mathematical problem},
where $f(x)\In Y$ is the set of solutions for an instance $x\in\dom(f)$.
Many mathematical problems can be naturally formalized and studied in this way.

We distinguish three different types of loops, which we refer to as {\em for loops}, {\em while loops}, and {\em infinite loops},
respectively. All of these types of loops and further constructions 
can be captured in Weihrauch complexity by certain operators, as indicated
in Figure~\ref{fig:operators-loops}.

\begin{figure}[h]
	\centering
	\begin{tabular}{l|l}
		{\bf operator} & {\bf loop}\\\hline
		$f^{[*]}$ & for loop\\
		$f^\diamond$ & while loop\\
		$f^\infty$ & infinite loop\\
		$\widehat{f}$ & parallelization
	\end{tabular}
	\caption{Operators and loops.}
	\label{fig:operators-loops}
\end{figure}

% Fix Indentation
\noindent
The operator of {\em parallelization} $f\mapsto\widehat{f}$ was introduced by Gherardi and the author~\cite{BG11} and was widely studied.
Here $\widehat{f}:=\bigtimes_{i\in\IN}f$ stands for the countable parallel application of $f$.
The {\em diamond operator} $f\mapsto f^\diamond$ was introduced by Neumann and Pauly~\cite{NP18}, inspired by the concept of generalized Weihrauch
reducibility, as introduced by Hirschfeldt and Jockusch~\cite{HJ16}.
The diamond operator was subsequently characterized by Westrick~\cite{Wes21}. It reflects a while loop in the sense that it allows for
arbitrarily but finitely many consecutive applications of the problem $f$ in a run of an algorithm.
The number of applications of $f$ is only determined in the course of the computation and not known beforehand.
By $f^{[n]}$ we denote the power of a computation, which can use $f$ $n$--times consecutively, i.e., informally
\[f^{[n]}:=\underbrace{f\star ...\star f}_{n\rm-times}.\]
Here $f\star g$ denotes the {\em compositional product} of $f$ and $g$, which was introduced by  Pauly and the author~\cite{BP18}.
It reflects the power of an algorithm, which first uses $g$ and then $f$.
One can then consider the coproduct 
\[f^{[*]}:=\bigsqcup_{n\in\IN}f^{[n]},\] 
which reflects the power of a {\em for loop}, as the number of applications of $f$
has to be determined a priori.
Finally, the operator $f\mapsto f^\infty$, which we call {\em infinite loop} or {\em inverse limit}, is a new one that we introduce here and 
that intuitively corresponds to an infinite loop that can use the problem $f$, i.e., informally
\[f^\infty:=\underbrace{\quad...\star f \star f}_{\mbox{\tiny countably many times}}.\]
More precise definitions follow below.
The requirement to use infinite loops occurs often when one computes sequences inductively. 
One setting where this arises naturally, is in solving initial value problems on their maximal domains of existence~\cite{BS25} or also in the proof of the Hahn-Banach theorem~\cite{BS26}.

One question we study here is which classes of problems are closed under the respective loops.
Many classes of problems can be characterized as lower cones of some problem in the Weihrauch lattice.
For instance, $f$ is {\em non-deterministically computable} in the sense of Ziegler~\cite{Zie07} if $f\leqW\C_{2^\IN}$ holds, i.e.,
if $f$ is Weihrauch reducible to choice on Cantor space, which in turn is equivalent to weak K\H{o}nig's lemma $\WKL$.
Precise definitions can be found below and in \cite{BGP21}.
The following table summarizes some of the closure properties of certain classes (a ``$+$'' indicates closure, a ``$-$'' indicates non-closure):

\begin{figure}[htb]
\centering
\begin{tabular}{l|c|cccc}
{\bf class of problems} & {\bf cone} & {\bf for} & {\bf while} & {\bf infinite} & {\bf parallel}\\
& $f$ & $f^{[*]}$ & $f^\diamond$ & $f^\infty$ & $\widehat{f}$\\\hline
computable & $\id$ & $+$ & $+$ & $+$ & $+$ \\
finite mind-change computable & $\C_\IN$ & $+$ & $+$ & $-$ & $-$ \\
non-deterministically computable & $\C_{2^\IN}$ & $+$ & $+$ & $+$ & $+$ \\
limit computable & $\lim$ & $-$ & $-$ & $-$ & $+$ \\
Borel computable & $\C_{\IN^\IN}$ & $+$ & $+$ & $+$ & $+$
\end{tabular}
\caption{Classes and closure properties.}
\end{figure} 

% Fix Indentation
\noindent
Most of the indicated closure properties can be seen straightforwardly or they follow from known results~\cite{BGP21}.
In this article we focus on the results regarding {\em infinite loops} and, in particular, we prove the following theorem
that yields the results in the corresponding column of the table.

\begin{thm}[Infinite loops]
\label{thm:infinite-loops}
We obtain:
\begin{enumerate}
\item $\C_{2^\IN}^\infty\equivW\C_{2^\IN}$, $\C_{\IN^\IN}^\infty\equivW\C_{\IN^\IN}$,
\item $\C_{\IN}\lW\lim\equivW\C_{\IN}^\infty\lW\lim^\infty$.
\end{enumerate}
\end{thm}

% Fix Indentation
\noindent
Beyond these particular applications, we also seek a more general understanding of the infinite loop operator $f \mapsto f^\infty$ and its relationship with other known operators. One interesting question in this context is under
which conditions the infinite loop is just the composition of the diamond operator $f\mapsto f^\diamond$ 
followed by the parallelization operator $f\mapsto\widehat{f}$.

\begin{question}
\label{question}
Characterize (classes of) problems $f$ for which $\widehat{f^\diamond}\equivW f^\infty$ holds!
\end{question}

% Fix Indentation
\noindent
Phrased differently, the question is whether infinite loops are more powerful than parallelized while loops,
and if so, for which problems?
Indeed, it turns out that for many problems $f$ the equivalence stated in the question is actually satisfied.
However, this is not always the case.

The structure of this article is as follows. 
In Section~\ref{sec:recursion}, we prove an injective version of the recursion theorem, a key tool that allows us to program infinite loops effectively. 
Section~\ref{sec:definitions} then establishes precise definitions of the infinite loop and the diamond operator (as well as recalling other standard notations). Section~\ref{sec:basic} explores basic properties of the infinite loop, including its relation to parallelization and the diamond operator. Section~\ref{sec:choice} focuses on closure properties of certain choice operations under infinite loops, demonstrating, in particular, that weak K\H{o}nig's lemma is closed under infinite loops. Finally, Section~\ref{sec:NON} discusses loops in settings where the underlying domain is the set of Turing degrees, illustrating some peculiarities that arise in that context. In particular, we show that, in general, the composition of the parallelization operator
and the diamond operator is weaker than the infinite loop operator.

\section{Injective recursion theorem}
\label{sec:recursion}

In this section, we state and prove an injective version of the recursion theorem that will be crucial for our analysis of infinite loops. We begin by recalling some necessary preliminaries. They 
are presented in greater detail in \cite{Bra23}.
A function $F:\In\IN^\IN\to\IN^\IN$ is {\em computable}, if there is some computable monotone word function 
${f:\IN^*\to\IN^*}$
that {\em approximates} $F$ in the sense that $F(p)=\sup_{w\prefix p}f(w)$ holds for all $p\in\dom(F)$. 
Likewise, $F$ is continuous if and only if an analogous condition holds for an arbitrary monotone word function $f$.
Using this characterization, we can define a representation $\Phi$ of the set $\CC(\In\IN^\IN,\IN^\IN)$ of certain continuous functions 
$F:\In\IN^\IN\to\IN^\IN$ (with natural domains) by encoding graphs of monotone word functions $f$ into names of $F$. 
For details see \cite{Wei87}.
Now we can define a computable {\em universal function} 
\[\U:\In\IN^\IN\to\IN^\IN,\langle q,p\rangle\mapsto\Phi_q(p)\] 
for all $p,q\in\IN^\IN$ \cite[Theorem~3.2.16~(1)]{Wei87}. 
For simplicity we will write $\U_q=\Phi_q$ in the following.
Here $\langle q,p\rangle:=q(0)p(0)q(1)p(1)...$ denotes the standard {\em pairing function} on Baire space.
Weihrauch~\cite[Theorems~3.5, 2.10, Corollary~2.11]{Wei85} (see also \cite[Theorem~3.2.16]{Wei87}) proved the following version of the smn-theorem.

\begin{thm}[smn]
\label{thm:smn}
For every computable (continuous) function ${F:\In\IN^\IN\to\IN^\IN}$ 
there exists a computable (continuous) total function
$S:\IN^\IN\to\IN^\IN$ such that $\U_{S(q)}(p)=F\langle q,p\rangle$ for all $\langle q,p\rangle\in\dom(F)$.
\end{thm}

%Fix Indentation
\noindent
Using the smn-theorem one can prove the following uniform version of the recursion theorem along the same lines as the classical recursion theorem.
It is an immediate corollary of a more general result due to Kreitz and Weihrauch~\cite[Theorem~3.4]{KW85}  (see also \cite[Theorem~3.3.20]{Wei87}).

\begin{thm}[Uniform recursion theorem]
\label{thm:recursion-theorem}
There exists a total computable function ${T:\IN^\IN\to\IN^\IN}$ such that 
$\U_{T(p)}=\U_{\U_pT(p)}$
for all $p\in\IN^\IN$ such that $\U_p$ is total.
\end{thm}

% Fix Indentation
\noindent
Our goal in this section is to prove a version of the recursion theorem, which simultaneously yields a
version of the smn-theorem with a computable injection. 
We recall that $F:\In\IN^\IN\to\IN^\IN$ is called a {\em computable injection}, if it is computable
and injective and there is a computable function $G:\In\IN^\IN\to\IN^\IN$ such that
$G\circ F(p)=p$ for all $p\in\dom(F)$.
In the following lemma we prove that program transformations can always be computably turned
into computable injections. 
Intuitively speaking, we can always encode the input of the program transformation
as a ``comment'' into the program text without changing the semantics of the program.
This is made formal by the following lemma.
For this purpose we assume, without loss of generality, that the descriptions $q$ of functions $\U_q=F$ do 
allow the digits $0,1,2$ as dummy symbols, i.e., adding or removing these digits does not change 
the meaning of such a name $q$.

\begin{lem}[Injection]
\label{lem:injection}
There is a total computable function $I:\IN^\IN\to\IN^\IN$ such that
$\U_{I(s)}:\IN^\IN\to\IN^\IN$ is a total computable injection for all $s\in\IN^\IN$ and
\[\U_{\U_{I(s)}(p)}=\U_{\U_s(p)}\]
for all $s,p\in\IN^\IN$ such that $\U_s(p)$ is defined. In fact, there is a single computable
function $L:\In\IN^\IN\to\IN^\IN$ such that $L\circ\U_{I(s)}=\id$ for all $s\in\IN^\IN$.
\end{lem}
\begin{proof}
If $\U_s(p)$ is defined we can read it as a description of a continuous function $\U_{\U_s(p)}$.
We now describe the computation of a total function $F:\IN^\IN\times\IN^\IN\to\IN^\IN$ that
yields a new description $F(s,p)$ of the same function, i.e., such that $\U_{\U_s(p)}=\U_{F(s,p)}$
whenever $\U_s(p)$ is defined. 
We can use the dummy symbols
$0$ and $1$ to encode $p$, e.g., by adding blocks of the form $10^{p(i)}1$ for $i=0,1,2,...$ to the encoded list.
Simultaneously, we can ensure that there are no other occurrences of the digits $0$ and $1$ by replacing
any of those by the digit $2$.
This describes how we can compute a new list $F(s,p)$, given $\U_s(p)$. 
This construction even works if $\U_s(p)$ is undefined (in which case the list will only contain the above blocks of $0,1$ from
a certain point on). 
This construction ensures that
$F$ is total computable and injective in the second component $p$. In fact, as a function of $p$ 
it is a computable injection, as we can extract $p$ from the list $F(s,p)$ computably. 
This extraction is described by a fixed computable function $L:\In\IN^\IN\to\IN^\IN$, not dependent on $s\in\IN^\IN$.
By the smn-theorem (Theorem~\ref{thm:smn}) 
there is 
a total computable $I:\IN^\IN\to\IN^\IN$ such that $\U_{I(s)}(p)=F(s,p)$ for all $s,p\in\IN^\IN$.
Altogether, this proves the claim.
\end{proof}

% Fix Indentation
\noindent
Now we are prepared to prove our injective recursion theorem. We recall that computation on the space
$\CC(\In\IN^\IN,\IN^\IN)$ is understood with respect to the representation $\Phi$.

\begin{thm}[Injective recursion theorem]
\label{thm:injective-recursion}
Let $f:\In\CC(\In\IN^\IN,\IN^\IN)\times\IN^\IN\to\IN^\IN$ be a computable function.
Then there is a total computable injection $R:\IN^\IN\to\IN^\IN$ such that
\[\U_{R(q)}(p)=f(R,\langle q,p\rangle)\]
for all $q,p\in\IN^\IN$ such that $f(R,\langle q,p\rangle)$ is defined.
\end{thm}
\begin{proof}
Let $I:\IN^\IN\to\IN^\IN$ be the computable function from Lemma~\ref{lem:injection}.
By a double application of the smn-theorem (Theorem~\ref{thm:smn}) there is a total computable ${S:\IN^\IN\to\IN^\IN}$ such that
\[\U_{\U_{S(s)}(q)}(p)=f(\U_{I(s)},\langle q,p\rangle)\]
for all $s,q,p\in\IN^\IN$, such that the right-hand side exists. 
Let $t\in\IN^\IN$ be such that $S=\U_t$ and let $T:\IN^\IN\to\IN^\IN$ be the computable
function from the uniform recursion theorem (Theorem~\ref{thm:recursion-theorem}).
Then $R:=\U_{IT(t)}$ is a total computable injection by Lemma~\ref{lem:injection} and we obtain
\[\U_{T(t)}=\U_{\U_tT(t)}=\U_{ST(t)}\]
and hence
\begin{eqnarray*}
\U_{R(q)}(p) 
&=&\U_{\U_{IT(t)}(q)}(p)=\U_{\U_{T(t)}(q)}(p)=\U_{\U_{ST(t)}(q)}(p)\\
&=&f(\U_{IT(t)},\langle q,p\rangle)=f(R,\langle q,p\rangle)
\end{eqnarray*}
for all $q,p\in\IN^\IN$ for which $f(R,\langle q,p\rangle)$ is defined.
\end{proof}

\section{Infinite loops and diamonds}
\label{sec:definitions}

In this section we provide the exact definition of the infinite loop operation and the diamond operator on problems. 
We also introduce some concepts from computable analysis and Weihrauch complexity
and we refer the reader to \cite{BH21,Wei00} for all concepts that have not been introduced here.
We follow the representation based approach to computable analysis and we recall
that a {\em representation} of a space $X$ is a surjective partial map $\delta_X:\In\IN^\IN\to X$.
In this case $(X,\delta_X)$ is called a {\em represented space}.
A function $F:\In\IN^\IN\to\IN^\IN$
is called a {\em realizer} of some partial multivalued function $f:\In X\mto Y$ on represented spaces $(X,\delta_X)$ and $(Y,\delta_Y)$, if
\[\delta_YF(p)\in f\delta_X(p)\]
for all $p\in\dom(f\delta_X)$. In this situation we also write $F\vdash f$.
A multivalued map $f:\In X\mto Y$ on represented spaces is called a {\em problem}, if it has a realizer.

We recall that the {\em composition} $g\circ f:\In X\mto Z$ of two problems $f:\In X\mto Y$ and $g:\In Y\mto Z$
is defined by 
\[g\circ f(x):=\{z\in Z:(\exists y\in f(x))\;z\in g(y)\}\]
with $\dom(g\circ f):=\{x\in \dom(f):f(x)\In\dom(g)\}$.

For simplicity we describe some constructions only for problems $f:\In\IN^\IN\mto\IN^\IN$ on Baire space.
All definitions can be generalized to arbitrary problems ${f:\In X\mto Y}$ on represented spaces $(X,\delta_X)$ and $(Y,\delta_Y)$ using standard methods
via the {\em realizer version} of $f$, defined by $f^\r:=\delta_Y^{-1}\circ f\circ\delta_X$.

Firstly, we recall the definition of $f\star g$ for problems $f,g:\In\IN^\IN\mto\IN^\IN$ from \cite{BGP21}:
\[f\star g:=\langle \id\times f\rangle \circ \U\circ\langle \id\times g\rangle.\]
We can define the infinite tupling function 
$\langle p_0,p_1,p_2,...\rangle$ and finite tupling functions of higher arity similarly as the pairing function.
Now we can define $f^\infty$ in a similar vein as the compositional product.
The definition is best understood as an inverse limit construction.

\begin{defi}[Infinite loops]
Let $f:\In\IN^\IN\mto\IN^\IN$ be a problem. Then we define
the {\em infinite loop} or {\em inverse limit} $f^\infty:\In\IN^\IN\mto\IN^\IN$ of $f$
by
\[f^\infty(q_0):=\{\langle q_0,q_1,q_2,...\rangle\in\IN^\IN:(\forall i)\;q_{i+1}\in \U\circ\langle \id\times f\rangle(q_i)\}\]
where $\dom(f^\infty)$ consists of all $q_0\in\IN^\IN$ such that
$A_0:=\{q_0\}\In\dom(\U\circ\langle\id\times f\rangle)$ and $A_{i+1}:=\U\circ\langle \id\times f\rangle(A_i)\In\dom(\U\circ\langle\id\times f\rangle)$ for all $i\in\IN$.
For an arbitrary problem $f:\In X\mto Y$ we define $f^\infty:=(f^\r)^\infty$.
\end{defi}

% Fix Indentation
\noindent
That is, the result $\langle q_0,q_1,q_2,...\rangle$ can be seen as the list of intermediate results
that one obtains if the infinite compositional product $...\star f\star f$ is evaluated on input $q_0$.
The domain $\dom(f^\infty)$ of $f^\infty$ is the largest set of inputs for which every infinite run is defined, independently of intermediate choices

We want to rephrase the definition of the diamond operator (see Westrick~\cite{Wes21} for a characterization) in similar terms.
To this end, the following terminology is useful. 
We call $\langle q_0,q_1,...,q_k\rangle$ a {\em finite run} of the loop on $f$,
if 
\[(\forall i<k)\;q_{i+1}\in \U\circ\langle \id\times f\rangle(q_i).\]
Likewise, we define an {\em infinite run} $\langle q_0,q_1,...\rangle$ with ``$(\forall i)$'' instead of ``$(\forall i<k)$''.
Using this terminology, we have
\[f^\infty(q_0)=\{\langle q_0,q_1,q_2,...\rangle\in\IN^\IN:\langle q_0,q_1,q_2,...\rangle\mbox{ is an infinite run on $f$}\}.\]
We say that a finite run $\langle q_0,q_1,...,q_k\rangle$ is {\em successful}, if
\[q_k(0)=0\mbox{ and }(\forall i<k)\;q_i(0)\not=0\]
and we say that an infinite run $\langle q_0,q_1,...\rangle$ is {\em unsuccessful}, if $(\forall i)\;q_i(0)\not=0$.
Intuitively speaking, we use the condition $q_k(0)=0$ to indicate that the run has come to a successful end.\footnote{We 
recall that $0,1,2$ are dummy symbols in the names $q\in\IN^\IN$ of continuous functions $\U_q$. Hence we can use, for instance $q(0)=0$ or $q(0)=1$ to indicate success without interference with the
meaning of the function $\U_q$.}
We say that the run $\langle q_0,q_1,...,q_k\rangle$ {\em stalls} if 
\[q_k\not\in\dom(\U\circ\langle\id\times f\rangle)\mbox{ and }(\forall i\leq k)\;q_i(0)\not=0.\]

% Fix Indentation
\noindent
Now we can define the diamond operator using this terminology as well. 

\begin{defi}[Diamond operator]
Let $f:\In\IN^\IN\mto\IN^\IN$ be a problem. Then we define
the {\em diamond operator} $f^\diamond:\In\IN^\IN\mto\IN^\IN$ of $f$
by
\[f^\diamond(q_0):=\{q_k\in\IN^\IN:(\exists q_1,...,q_{k-1}\in\IN^\IN)\;\langle q_0,...,q_k\rangle\mbox{ is a successful finite run on $f$}\}\]
where $\dom(f^\diamond)$ consists of all $q_0\in\IN^\IN$ such that there is no run starting with $q_0$ that stalls or is infinite and unsuccessful.
\end{defi}

% Fix Indentation
\noindent
Again the definition can be extended to arbitrary problems $f$ using their realizer version $f^\r$.
We point out that there is a formal similarity between the definition of the diamond operator and the $\mu$--operator
from classical computability theory~\cite{Odi89}. Both constructions reflect the power of while loops:
the $\mu$--operator does so for single-valued computations on the natural numbers and the 
diamond operator ${f\mapsto f^\diamond}$
for arbitrary multivalued problems with $f$ as a subroutine.

Problems can be compared using the tool of (strong) Weihrauch reducibility~\cite{BGP21}.
By $\id:\IN^\IN\to\IN^\IN$ we denote the {\em identity} on Baire space.

\begin{defi}[Weihrauch reducibility]
Let $f:\In X\mto Y$ and $g:\In Z\mto W$ be problems. We say that
\begin{enumerate}
\item $f$ is {\em Weihrauch reducible} to $g$, in symbols $f\leqW g$, if there are computable functions
         ${H,K:\In\IN^\IN\to\IN^\IN}$ such that $H\langle\id,GK\rangle\vdash f$, whenever $G\vdash g$ holds. 
\item $f$ is {\em strongly Weihrauch reducible} to $g$, in symbols $f\leqSW g$, if there are computable
         $H,K:\In\IN^\IN\to\IN^\IN$ such that $HGK\vdash f$, whenever $G\vdash g$ holds. 
\end{enumerate}
\end{defi}

% Fix Indentation
\noindent
As usual, we denote the corresponding equivalences by $\equivW$ and $\equivSW$, respectively. 
We recall that a problem $f$ is called {\em pointed} if $\id\leqW f$ holds, i.e., if and only if $f$ has a computable input.
The problem $f^\diamond$ is always pointed, as the zero input is a successful run in which no input for $f$ is required.

The main result of Westrick~\cite{Wes21} regarding the diamond operator is the following theorem
that characterizes the diamond operator as a closure operator that reflects closure under
compositional product.

\begin{thm}[Westrick 2021]
\label{them:Westrick}
For each pointed problem $f$ we have
\[f^\diamond\equivW\min\nolimits_{\leqW}\{g:f\leqW g\star g\leqW g\}.\]
\end{thm}

% Fix Indentation
\noindent
It is clear that {\em while loops} can be used to simulate {\em for loops}. 
The success condition used in a while loop can simply be that a given number
of runs of the loop is performed.

\begin{prop}
\label{prop:for-while}
$f^{[*]}\leqSW f^\diamond$ for all problems $f$.
\end{prop}

% Fix Indentation
\noindent
We close this section with mentioning a number of standard problems that we are going to use in the
following (see~\cite{BGP21} for more precise definitions). 
By $\C_X$ we denote the choice problem of a computable metric space $X$, which is defined
by $\C_X:\In\AA_-(X)\mto X,A\mapsto A$, where $\AA_-(X)$ denotes the set of closed subsets of $X$
given by negative information. 
By $\U\C_X$ we denote the {\em unique choice problem}, which is the restriction of $\C_X$ to singletons.
The problem $\LLPO:=\C_2=\C_{\{0,1\}}$ is also known as
{\em lesser limited problem of omniscience}. The problem $\LPO:\IN^\IN\to\{0,1\}$ is simply
the characteristic function of $\{\widehat{0}\}$, where $\widehat{0}$ denotes the constant zero sequence.
By $\lim_X:\In X^\IN\to X$ we denote the usual {\em limit map} of a metric space $X$, where $\lim:=\lim_{\IN^\IN}$
stands for the limit of Baire space. 
By $\J:\IN^\IN\to\IN^\IN,p\mapsto p'$ we denote the {\em Turing jump operator}.
The problem $\WKL$ stands for {\em weak K\H{o}nig's lemma} and it
is the problem $\WKL:\In\Tr_2\mto 2^\IN, T\mapsto [T]$ that maps every infinite binary tree $T$ to the set
of its infinite paths. The following proposition summarizes some well-known results about some of these problems~\cite{BG11,BBP12,BGP21}.

\begin{prop}
\label{prop:classical-problems}
We obtain 
\begin{enumerate}
\item $\widehat{\LPO}\equivSW\widehat{\C_\IN}\equivSW\lim\equivSW\J$, 
\item $\C_\IN^\diamond\equivSW\C_\IN\equivSW\lim_\IN$, and  
\item $\widehat{\LLPO}\equivSW\C_{2^\IN}\equivSW\WKL$.
\end{enumerate}
\end{prop}

\section{Basic properties of infinite loops}
\label{sec:basic}

In this section we prove some basic properties  of the infinite loop operation.
Programming with infinite loops is not so straight-forward because the ``program'' to which $\U$
is applied in each loop needs to be inherited from the previous loop.
This is exactly what can be achieved with the injective recursion theorem.
We work this out in technical detail in some of our proofs, but leave the technical details
to the reader for most of the others.

We start by showing that $f\mapsto f^\infty$ is actually an operation on (strong) Weihrauch degrees.
In fact, the problems $f^\infty$ are all {\em cylinders} (i.e., $f^\infty\equivSW\id\times f^\infty$), 
hence we always get strong Weihrauch reductions. We write $g\prefix f$ for two problem $f,g:\In\IN^\IN\mto\IN^\IN$
if $\dom(f)\In\dom(g)$ and $g(p)\In f(p)$ for all $p\in\dom(f)$.

\begin{prop}[Monotonicity of infinite loops]
\label{prop:monotonicity}
$f\leqW g\TO f^\infty\leqSW g^\infty$ holds for all problems $f,g$.
\end{prop}
\begin{proof}
Without loss of generality, we can assume that $f,g:\In\IN^\IN\mto\IN^\IN$ are problems on Baire space.
Let $f\leqW g$. Then there are computable functions $H,K:\In\IN^\IN\to\IN^\IN$ such that
$H\circ\langle\id, gK\rangle\prefix f$. 
Let $K_2:\In\IN^\IN\to\IN^\IN$ be the computable function
with $K_2\langle q,p\rangle:=K(p)$. 
By the injective recursion theorem (Theorem~\ref{thm:injective-recursion}) 
there exists a total computable injection $K_1:\IN^\IN\to\IN^\IN$ such that
\[\U_{K_1\langle q,p\rangle}(r)=\langle K_1, K_2\rangle\circ\U_q \circ H\langle p,r\rangle\]
for all $q,p,r\in\IN^\IN$ such that the right-hand side exists.
That $K_1$ is a computable injection means that there is a computable function $H_1:\In\IN^\IN\to\IN^\IN$ such that $H_1\circ K_1\langle q,p\rangle=\langle q,p\rangle$
for all $\langle q,p\rangle\in\IN^\IN$ and we obtain
\begin{eqnarray*}
\U\circ\langle\id\times g\rangle\circ \langle K_1, K_2\rangle\langle q,p\rangle
&=&\U_{K_1\langle q,p\rangle}(gK(p))\\
&=&\langle K_1,K_2\rangle\circ\U_q\circ H\langle p,gK(p)\rangle\\
&\In&\langle K_1, K_2\rangle\circ \U\circ\langle\id\times f\rangle\langle q,p\rangle.
\end{eqnarray*}
Let $H_2:\In\IN^\IN\to\IN^\IN$ be the computable function with 
\[H_2\langle\langle q_0,p_0\rangle,\langle q_1,p_1\rangle,...\rangle:=\langle H_1(q_0),H_1(q_1),...\rangle.\]
Then we obtain 
\[H_2\circ g^\infty\circ\langle K_1,K_2\rangle\prefix H_2\circ\widehat{\langle K_1,K_2\rangle}\circ f^\infty=f^\infty\]
and thus $f^\infty\leqSW g^\infty$.
\end{proof}

% Fix Indentation
\noindent
With the next result we want to establish some facts on the relation between the diamond operator
and infinite loops. In this case we leave the reasoning informal and do not work out the 
technical details of the application of the inverse recursion theorem, as it would be very
technical and block the view on the essential ideas.

\begin{prop}[Parallelization, infinite loops and diamonds]
\label{prop:parallelization-diamond}
For arbitrary problems $f$ we obtain:
\begin{enumerate}
\item $\widehat{f^\infty}\equivSW f^\infty$,
\item $f^\diamond\leqSW f^\infty$, if $f$ is pointed,
\item $\widehat{f^\diamond}\leqSW f^\infty$, if $f$ is pointed.
\end{enumerate}
\end{prop}
\begin{proof}
(1) We can obtain $f^\infty(p_i)$ on countably many inputs $p_0,p_1,p_2,...$ in parallel
by a single application of $f^\infty$ as follows: we use the $\langle i,n\rangle$--th application
of $f$ to simulate the $n$--th application of $f$ on input $p_i$. In this way we can
obtain all the results in parallel. This idea can be implemented with the help
of the injective recursion theorem.
The inverse reduction holds obviously, as $f\leqSW\widehat{f}$ for every problem $f$.\\
(2) We can obtain $f^\diamond(p)$ by $f^\infty$ on a suitable input as follows:
we apply $f$ as often as is necessary until we have a successful finite run
for $f^\diamond(p)$, then we add ``redundant runs'' of $f$ on some fixed input
from the domain of $f^\infty$ (which is possible as $f$ is pointed).
From the result we can read off some value for $f^\diamond(p)$. 
Again, this idea can be implemented with the help of the injective recursion theorem.\\
(3) This is just a consequence of (1) and (2).
\end{proof}

% Fix Indentation
\noindent
The diagram in Figure~\ref{fig:operators} illustrates the situation for pointed problems.

\begin{figure}[htb]
\begin{center}
\begin{tikzpicture}[scale=1,auto=left]
\node[style={fill=blue!20}]  (id) at (0,0) {$\id$};
\node[style={fill=blue!20}]  (f) at (0,1) {$f$};
\node[style={fill=blue!20}]  (pf) at (-1,2) {$\widehat{f}$};
\node[style={fill=blue!20}]  (df) at (1,3) {$f^\diamond$};
\node[style={fill=blue!20}]  (sf) at (1,2) {$f^{[*]}$};
\node[style={fill=blue!20}]  (of) at (-1,3) {$f^\omega$};
\node[style={fill=blue!20}]  (pdf) at (0,4) {$\widehat{f^\diamond}$};
\node[style={fill=blue!20}]  (if) at (0,5) {$f^\infty$};
\draw[->] (if) edge (pdf);
\draw[->] (pdf) edge (of);
\draw[->] (of) edge (pf);
\draw[->] (of) edge (sf);
\draw[->] (pdf) edge (df);
\draw[->] (pf) edge (f);
\draw[->] (df) edge (sf);
\draw[->] (sf) edge (f);
\draw[->] (f) edge (id);
\end{tikzpicture}
\caption{Operators on pointed problems.}
\label{fig:operators}
\end{center}
\end{figure}

% Fix Indentation
\noindent
Another operation related to loops can be defined by $f^\omega:=\bigtimes_{n\in\IN}f^{[n]}$.
For pointed problems we will show that this is equivalent to the
parallelization of $f^{[*]}$. 
For simplicity, we 
consider this auxiliary operation only for problems on Baire space.
Using the realizer version $f^\r$ of a problem $f$ we can extend everything to arbitrary problems.

\begin{defi}[Repeated compositional products]
We define
\begin{enumerate}
\item $f^{[0]}:=\id$, $f^{[1]}:=\langle\id\times f\rangle$ and 
\item $f^{[n+1]}:=\langle\id\times f\rangle\circ\U\circ f^{[n]}$ 
\end{enumerate}
for all problems $f:\In\IN^\IN\mto\IN^\IN$ and $n\geq1$.
\end{defi}

% Fix Indentation
\noindent
Using this definition, we can now define $f^\omega$ as follows.

\begin{defi}[Omega operation]
For every problem $f:\In\IN^\IN\mto\IN^\IN$ we define the problem ${f^\omega:\In\IN^\IN\mto\IN^\IN}$ by
\[f^\omega(p):=\langle f^{[0]}(p),f^{[1]}(p),f^{[2]}(p),...\rangle\]
for all $p\in\IN^\IN$ such that $p\in\dom(f^{[n]})$ for all $n\in\IN$.
\end{defi}

% Fix Indentation
\noindent
Again, this definition can be extended to arbitrary problems $f$ via their realizer version $f^\r$.
We note that we could equivalently define 
\[f^\omega\langle p_0,p_1,p_2,...\rangle:=\langle f^{[0]}(p_0),f^{[1]}(p_1),f^{[2]}(p_2),...\rangle.\]
For pointed $f$ this makes no essential difference, as we could always give a dummy input to the first $f$
and use the program input to $\U$ to extract components of an input $p=\langle p_0,p_1,p_2,...\rangle$ step by step as required.

\begin{prop}[Omega operation]
\label{prop:omega}
For problems $f:\In\IN^\IN\mto\IN^\IN$ we obtain
\begin{enumerate}
\item $\widehat{f^\omega}\equivSW f^\omega\leqSW\widehat{f^\diamond}$,
\item $f^{[*]}\leqSW f^\omega$, if $f$ is pointed,
\item $\widehat{f^{[*]}}\equivSW f^\omega$, if $f$ is pointed.
\end{enumerate}
\end{prop}
\begin{proof}
(1) It is easy to see that $f^{[n]}\leqSW f^\diamond$, as we can create a program that runs $f$ exactly $n$--times with the help of $f^\diamond$ before it comes to a successful halt.
This reduction can even be made uniform in $n$, which allows us to conclude $f^\omega\leqSW\widehat{f^\diamond}$.
It is also easy to see that $f^\omega$ is strongly parallelizable, i.e., $\widehat{f^\omega}\leqSW f^\omega$. 
This is because we have $f^{[n]}\leqSW f^{[k]}$ for $n\leq k$ uniformly in $n$ and $k$.\\
(2) Since $f^{[n]}\leqSW f^\omega$ holds for pointed $f$ uniformly in $n\in\IN$, we obtain $f^{[*]}\leqSW f^\omega$.\\
(3) By (1) this implies $\widehat{f^{[*]}}\leqSW f^\omega$. The inverse reduction is clear.
\end{proof}

% Fix Indentation
\noindent
One might be tempted to believe that $f^\omega\equivSW\widehat{f^\diamond}$. However, this is only true in certain cases,
for instance for single-valued $f$ (see Corollary~\ref{cor:single-valued}) or for problems on Turing degrees (see Proposition~\ref{prop:Turing}).
In order to construct a counterexample for the general case, we use the following lemma.
It is well-known that there are strictly descending chains of Turing degrees whose maximum is strictly above the remainder of the sequence.

\begin{lem}
\label{lem:Turing}
There exists a sequence $(p_n)_{n\in\IN}$ in $\IN^\IN$ such that 
$p_{n+1}\lT p_n$ for all $n\in\IN$ and $\langle p_1,p_2,p_3,...\rangle\lT p_0$.
\end{lem}

% Fix Indentation
\noindent
For instance, $0\lT \dots\lT p_3\lT p_2\lT p_1$ could be an initial
segment of the Turing degrees~\cite{Hug69} and
$p_0:=\langle p_1,p_2,p_3,\dots\rangle'$.
We use this fixed sequence for the next example and also for Example~\ref{example:infty-omega}.

\begin{exa}
\label{example:diamond-omega}
Let $\widehat{0}\in\IN^\IN$ denote the constant zero sequence. We consider the problem $f:\In\IN^\IN\mto\IN^\IN$
with $\dom(f)=\{\widehat{0}\}\cup\{p_n:n\in\IN\}$ with
\[f(p):=\left\{\begin{array}{ll}
0p_0 & \mbox{if $p=p_0$}\\
np_n & \mbox{if $p=p_{n+1}$}\\
\{np_n:n\in\IN\} & \mbox{if $p=\widehat{0}$}
\end{array}\right.\]
Then $f^\diamond\nleqW f^\omega$.
\end{exa}
\begin{proof}
We consider the problem $g:\IN^\IN\to\IN^\IN$ with $g(p)=p_0$. We claim that $g\leqW f^\diamond$.
We simply determine $np_n\in f(\widehat{0})$ for some $n\in\IN$ and then we need $n$ further applications of $f$, starting with $f(p_n)$,
in order to compute $p_0$. The first number in the output being $0$ indicates success of the computation.
On the other hand, $f^\omega(p)$ on some computable input $p$ has
a possible output of Turing degree equal to $p_1 \lT p_0$, as $f^{[n]}(p)$ has a possible output $kp_k$ for arbitrary $k\in\IN$, hence in particular one of Turing degree equal to $p_{1}$ for every $n\in\IN$. 
Thus, $f^\diamond\nleqW f^\omega$.
\end{proof}

% Fix Indentation
\noindent
This example together with the fact that 
\[\widehat{\C_\IN}\equivSW\lim\nleqW\C_\IN\equivSW\C_\IN^\diamond\mbox{ and }\lim\nolimits^{[2]}\nleqW\lim\equivSW\widehat{\lim}\]
shows that the diagram in Figure~\ref{fig:operators} does not allow for any further arrows that involve
$f^\diamond$ or $\widehat{f}$.

For single-valued problems $F:\In\IN^\IN\to\IN^\IN$ the infinite loop operation coincides with the omega operation.
For this result we work out the formal details and we demonstrate again how the injective recursion theorem can be used.

\begin{prop}[Infinite loops and the omega operation]
\label{prop:inverse-omega}
We obtain $F^\omega\equivSW F^\infty$ for all single-valued problems $F:\In\IN^\IN\to\IN^\IN$.
\end{prop}
\begin{proof}
$F^\infty\leqSW F^\omega$ follows for single-valued $F$ from
\[F^\infty(p)=\langle\id\times\U\times\U\times\U\times...\rangle\circ F^\omega(p).\]
We still need to prove $F^\omega\leqSW F^\infty$. 
For a function $K:\IN^\IN\to\IN^\IN$ we write $K_q(p):=K\langle q,p\rangle$ for all $q,p\in\IN^\IN$.
By the injective recursion theorem (Theorem~\ref{thm:injective-recursion}) there exists a total computable injection
$K:\IN^\IN\to\IN^\IN$ such that
\[\U_{K_{\langle s,t\rangle}(q)}(r)=\langle K_{\langle q,r\rangle}\times\id\rangle\circ\U\langle q,r\rangle\]
for all $s,t,q,r\in\IN^\IN$ such that the right-hand side exists.
Then we obtain
\[\U\circ\langle\id\times F\rangle\circ \langle K_{\langle s,t\rangle}\times\id\rangle\langle q,p\rangle
=\U_{\K_{\langle s,t\rangle}(q)}F(p)=\langle K_{\langle q,F(p)\rangle}\times\id\rangle\circ\U\circ\langle \id\times F\rangle\langle q,p\rangle\]
for all $s,t,q,p\in\IN^\IN$ such that the right-hand side exists. 
Together with the computable function $K_0:=\langle K_{\langle\widehat{0},\widehat{0}\rangle}\times\id\rangle$
we obtain inductively
\[\U\circ F^{[n]}\circ K_0\langle q,p\rangle=\langle K_{F^{[n]}\langle q,p\rangle}\times\id\rangle\circ\U\circ F^{[n]}\langle q,p\rangle\]
for all $q,p\in\IN^\IN$ and $n\geq1$ such that the right-hand side exists.
Let $L,L_0:\IN^\IN\to\IN^\IN$ be computable functions with $L\circ K_{\langle s,t\rangle}(p)=\langle s,t\rangle$ 
and $L_0\circ K_{\langle s,t\rangle}(p)=p$ for all
$s,t,p\in\IN^\IN$ and let $H:\IN^\IN\to\IN^\IN$ be the computable function defined by
\[H\langle\langle p_0,q_0\rangle,\langle p_1,q_1\rangle,\langle p_2,q_2\rangle,...\rangle:=\langle \langle L_0(p_0),q_0\rangle,L(p_1),L(p_2),L(p_3),...\rangle.\]
Now we obtain for suitable $r_n\in\IN^\IN$
\begin{eqnarray*}
&&H\circ F^\infty\circ K_0\langle q,p\rangle\\
&=& H\circ\langle K_0\langle q,p\rangle,\U\circ F^{[1]}\circ K_0\langle q,p\rangle,\U\circ F^{[2]}\circ K_0\langle q,p\rangle,\U\circ F^{[3]}\circ K_0\langle q,p\rangle,...\rangle\\
 &=& \langle \langle q,p\rangle,L\circ K_{F^{[1]}\langle q,p\rangle}(r_1),L\circ K_{F^{[2]}\langle q,p\rangle}(r_2),L\circ K_{F^{[3]}\langle q,p\rangle}(r_3),...\rangle\\
 &=& \langle F^{[0]},F^{[1]},F^{[2]},F^{[3]},...\rangle\langle q,p\rangle\\
 &=& F^\omega\langle q,p\rangle.
\end{eqnarray*}
That is $F^\omega\leqSW F^\infty$.
\end{proof}

% Fix Indentation
\noindent
We note that in the special case of pointed $f$ we can conclude $f^\omega\leqSW f^\infty$ also
from Propositions~\ref{prop:parallelization-diamond} and \ref{prop:omega}.
These propositions together with Proposition~\ref{prop:inverse-omega} also have the following consequence.

\begin{cor}[Infinite loops and single-valuedness]
\label{cor:single-valued}
$F^\infty\equivSW F^\omega\equivSW\widehat{F^\diamond}$ for single-valued pointed $F:\In\IN^\IN\to\IN^\IN$.
\end{cor}

% Fix Indentation
\noindent
We use this result in order to determine the infinite loop of $\LPO$ and the limit map $\lim$.
It turns out that $\lim^\infty$ is equivalent to the {\em $\omega$--Turing jump operator}
\[\J^{(\omega)}:\IN^\IN\to\IN^\IN,p\mapsto\langle p,p',p'',...\rangle,\]
which is well-known in computability theory.
In the following result we also mention the operator $f\mapsto f^\dagger$ that was introduced by Pauly~\cite{Pau15}
and allows parallelizations and compositions governed by ordinals.

\begin{prop}[Infinite loop of LPO and the limit map]
\label{prop:LPO-lim}
We obtain:
\begin{enumerate}
\item $\LPO^\infty\equivSW\C_\IN^\infty\equivSW\lim$,
\item $\lim^\infty\equivSW\J^{(\omega)}\lW\lim^\dagger\equivW\U\C_{\IN^\IN}$.
\end{enumerate}
\end{prop}
\begin{proof}
(1) It is clear that 
\[\lim\equivSW\widehat{\LPO}\leqSW\LPO^\infty\leqSW\C_\IN^\infty\equivSW\lim\nolimits_\IN^\infty\]
holds by
Propositions~\ref{prop:parallelization-diamond}. We still need to prove $\lim_\IN^\infty\leqSW\lim$.
Since $\lim$ is a cylinder, it suffices to show $\lim_\IN^\infty\leqW\lim$ to complete the proof.
To this end we need to show that $\lim_\IN^\infty$ is limit computable.
The output $r=\langle \langle q_0,p_0\rangle,\langle q_1, p_1\rangle,\langle q_2, p_2\rangle,...\rangle$ of 
$\lim_\IN^\infty$ upon input $\langle q_0,p_0\rangle$ can be
computed on a limit machine as follows. We inspect the sequence $p_0$  
that is given as input to the first $\lim_\IN$
and start to compute $\langle q_1,p_1\rangle$ on basis of the assumption that $p_0$ is constantly $p_0(0)$.
As soon as some value $p_1(0)$ is produced, we start computing $\langle q_2,p_2\rangle$ on basis of
the assumption that $p_1$ is constantly $p_1(0)$ and so forth. Simultaneously, we inspect the $p_i$
to check whether the assumptions were justified. As soon as we find a value $p_0(k)\not=p_0(0)$
we restart the entire process under the assumption that $p_0$ is constantly $p_0(k)$ from $k$ onwards.
After finitely many restarts, caused by $p_0$, we will not have to change our mind if $p_0\in\dom(\lim_\IN)$.
From this moment on $\langle q_1,p_1\rangle$ will be computed correctly and likewise we will have
to change our mind with regards to the assumption on $p_1$ at most finitely many times. 
We continue inductively like this in order to compute $r$ and any for any component $\langle q_i,p_i\rangle$
of it only finitely many mind changes are required before the computation eventually produces the correct result.\\
(2) Firstly, we note that $\J^{(\omega)}\equivSW\J^\omega$. Here $\J^{(\omega)}\leqSW\J^\omega$ follows,
as by the recursion theorem there exists a $q\in\IN^\IN$ with $\U_q(p)=\langle q,p\rangle$ for all $p\in\IN^\IN$.
Hence,
\[\J^{[n]}\langle q,p\rangle=\underbrace{\langle\id\times \J\rangle\circ\U\circ\langle \id\times \J\rangle\circ \U\circ....\circ\U\circ\langle\id\times\J\rangle}_{\mbox{\tiny $n$--many $\J$}}\langle q,p\rangle=\langle q,p^{(n)}\rangle\]
for all $p\in\IN^\IN$ and $n\in\IN$.
The inverse reduction follows as by \cite[Theorem~2.14]{Bra18} there is a computable sequence $(H_n)_{n\in\IN}$ of computable
functions $H_n:\In\IN^\IN\to\IN^\IN$ such that
\[\J^{[n]}=H_n\circ\underbrace{\J\circ...\circ\J}_{\mbox{\tiny $n$--times}}.\]
As $\J\equivSW\lim$, we obtain $\lim^\infty\equivSW\J^{(\omega)}$ by Proposition~\ref{prop:inverse-omega}.
By Corollary~\ref{cor:single-valued}
\[\mbox{$\lim^\infty\equivSW\lim^\omega\leqSW\lim^\dagger$}.\]
The equivalence $\lim^\dagger\equivW\U\C_{\IN^\IN}$ was proved in \cite[Theorem~80]{Pau15}.
By the same theorem it follows that $\J^{(\omega)}\lW\lim^\dagger$, as 
there are ordinals $\alpha$ with $\omega<\alpha$ such that we obtain $\J^{(\omega)}\lW\J^{(\alpha)}\leqW\J^\dagger\equivSW\lim^\dagger$, where $\J^{(\alpha)}$ denotes the $\alpha$--jump operator.
\end{proof}

% Fix Indentation
\noindent
From this result it follows, in particular, that $f\mapsto f^\infty$ is not a closure operator.
Secondly, we can conclude that the operators $f\mapsto\widehat{f}$, $f\mapsto f^\diamond$,
$f\mapsto f^\infty$ and $f\mapsto f^\dagger$ are pairwise different from each other.
In particular, the infinite loop is really a new operation.

\begin{cor}
We obtain:
\begin{enumerate}
\item $\LPO^\diamond\equivW\C_\IN\lW\lim\equivSW\LPO^\infty\lW\U\C_{\IN^\IN}\equivW \LPO^\dagger$,
\item $\widehat{\lim}\equivW\lim\lW\lim^\diamond\lW\lim^\infty$.
\item $\LPO^\infty\equivSW\lim\lW\lim^\infty\equivSW\LPO^{\infty\infty}$.
\end{enumerate}
\end{cor}

% Fix Indentation
\noindent
Here $\U\C_{\IN^\IN}\equivW\LPO^\dagger$ was proved in \cite[Theorem~80]{Pau15}
and $\LPO^\diamond\equivW\C_\IN$ was proved in \cite[Proposition~10]{NP18}.
Finally, $\lim^\infty\nleqW\lim^\diamond$ holds, as the former has a non-arithmetical output on
some computable inputs, whereas the latter always has arithmetical outputs on computable inputs.

\section{Choice and infinite loops}
\label{sec:choice}

Our next result describes the action of infinite loops with respect to certain closed choice problems.
This result can be seen as an infinitary version of the independent choice theorem \cite[Theorem~7.2]{BBP12}.
We use the {\em Sierpi\'nski space} $\IS=\{0,1\}$ with its usual representation $\delta_\IS$ given by
$\delta_\IS(p)=0\iff p=\widehat{0}$.

\begin{thm}[Countable independent choice]
\label{thm:independent}
Let $A\In\IN^\IN$ and let $f$ be a problem.
\begin{enumerate}
\item If $f\leqW\C_A$, then $f^\infty\leqW\C_{A^\IN}$.
\item If $f\leqW\U\C_A$, then $f^\infty\leqW\U\C_{A^\IN}$.
\end{enumerate}
\end{thm}
\begin{proof}
Since $f\equivSW f^\r$ it is sufficient by Proposition~\ref{prop:monotonicity} to prove the claim for problems $f:\In\IN^\IN\mto\IN^\IN$. We prove the first statement.
Let $f\leqW\C_A$. Then $f$ is non-deterministically computable with advice space $A$
according to \cite[Definition~7.1, Theorem~7.2]{BBP12}.  That is, there are two computable functions $F_1,F_2:\In\IN^\IN\to\IN^\IN$
with $\langle\dom(f)\times A\rangle\In\dom(F_2)$ and such that for each $p\in\dom(f)$ we have
\begin{enumerate}
\item $(\exists r\in A)\;\delta_\IS F_2\langle p,r\rangle=0$,
\item $(\forall r\in A)\;(\delta_\IS F_2\langle p,r\rangle=0\TO F_1\langle p,r\rangle\in f(p))$.
\end{enumerate}
That is, $f$ can be computed by $F_2$ with advice $r\in A$, provided the advice is helpful, where
$F_2$ can recognize a non-helpful advice.
The first condition guarantees that there is at least one helpful advice.

We need to show that $f^\infty$ is non-deterministically computable with advice space $A^\IN$.
In order to keep track of the individual components, we consider $f^\infty$ with input $\langle q_0,p_0\rangle$, i.e.,
\[f^\infty\langle q_0,p_0\rangle=\{\langle\langle q_0,p_0\rangle,\langle q_1,p_1\rangle,...\rangle:
(\forall i)\;\langle q_{i+1},p_{i+1}\rangle\in\U\circ\langle\id\times f\rangle\langle q_i,p_i\rangle\}\]
We need to show that there are computable functions $G_1,G_2:\In\IN^\IN\to\IN^\IN$ 
with the property
$\langle \dom(f^\infty)\times A^\IN\rangle\In\dom(G_2)$ and such that for each $\langle q_0,p_0\rangle\in\dom(f^\infty)$ we have
\begin{enumerate}
\item $(\exists r\in A^\IN)\;\delta_\IS G_2\langle\langle q_0,p_0\rangle,r\rangle=0$,
\item $(\forall r\in A^\IN)\;(\delta_\IS G_2\langle \langle q_0,p_0\rangle,r\rangle=0\TO G_1\langle \langle q_0,p_0\rangle,r\rangle\in f^\infty\langle q_0,p_0\rangle)$.
\end{enumerate}
We define the required computable function $G_1$ by
\[G_1\langle\langle q_0,p_0\rangle,\langle r_0,r_1,...\rangle\rangle:=\langle\langle q_0,p_0\rangle,\langle q_1,p_1\rangle,...\rangle\]
for $\langle q_0,p_0\rangle\in\IN^\IN$ and $r_i\in A$ if all the values
\[\langle q_{i+1},p_{i+1}\rangle:=\U\circ\langle\id\times F_1\rangle\langle q_i,\langle p_i,r_i\rangle\rangle\]
exist (and otherwise we leave $G_1$ undefined). There is also a corresponding computable $G_2$ that satisfies
the following properties for $r:=\langle r_0,r_1,r_2,...\rangle$:
\begin{itemize}
\item $(\forall i)(F_2\langle p_i,r_i\rangle=\widehat{0})\TO G_2\langle\langle q_0,p_0\rangle,r\rangle=\widehat{0}$,
\item $(\exists k)(F_2\langle p_k,r_k\rangle\in\IN^\IN\setminus\{\widehat{0}\}$ and $(\forall i<k)\;F_2\langle p_i,r_i\rangle=\widehat{0})$\\ $\TO G_2\langle\langle q_0,p_0\rangle,r\rangle\in\IN^\IN\setminus\{\widehat{0}\}$.
\end{itemize}

% Fix Indentation
\noindent
We need to show that $G_1$ and $G_2$ satisfy the claims. For this purpose, let $\langle q_0,p_0\rangle\in\dom(f^\infty)$.
Then for every $r=\langle r_0,r_1,...\rangle$ with $r_i\in A$  we have different cases that are to be considered.
Firstly, $\langle q_0,p_0\rangle\in\dom(\U\circ\langle\id\times f\rangle)$ and if $\delta_\IS F_2\langle p_0,r_0\rangle=0$, then $F_1\langle p_0,r_0\rangle\in f(p_0)$ and
hence $\langle q_1,p_1\rangle$ exists and is in $\dom(\U\circ\langle\id\times f\rangle)$. If we continue inductively like this and $\delta_\IS F_2\langle p_i,r_i\rangle=0$
for all $i$, then indeed $\delta_\IS G_2\langle\langle q_0,p_0\rangle,r\rangle=0$ and
$G_1\langle\langle q_0,p_0\rangle,r\rangle\in f^\infty\langle q_0,p_0\rangle$.
Since there is a suitable $r_i\in A$ with $\delta_\IS F_2\langle p_i,r_i\rangle=0$ for each $i$, it follows that,
in particular, there is some suitable $r\in A^\IN$ with $\delta_\IS G_2\langle\langle q_0,p_0\rangle,r\rangle=0$
and hence $G_2$ meets condition (1). Another case for $r$ that we need to consider is the case that
eventually there is a first $k$ such that $F_2\langle p_i,r_i\rangle$ still exists but is different from $\widehat{0}$.
In this case $G_2\langle\langle q_0,p_0\rangle,r\rangle$ exists and is also different from $\widehat{0}$.
Since $\langle \dom(f)\times A\rangle\In\dom(F_2)$, no other case needs to be considered and condition (2)
is met too and $\langle \dom(f^\infty)\times A^\IN\rangle\In\dom(G_2)$.

The statement for $f\leqW\U\C_A$ can be proved analogously, except that a helpful advice 
$r\in A$ and $r\in A^\IN$, respectively, is now unique.
\end{proof}

% Fix Indentation
\noindent
Using $\C_{(A^\IN)^\IN}\equivSW\C_{A^\IN}$, $\U\C_{(A^\IN)^\IN}\equivSW\U\C_{A^\IN}$, 
we directly obtain the following corollary.

\begin{cor}
For every $A\In\IN^\IN$ we obtain:
\begin{enumerate}
\item $\C_A^\infty\leqW\C_{A^\IN}$ and $\C_{A^\IN}^\infty\equivW\C_{A^\IN}$,
\item $\U\C_A^\infty\leqW\U\C_{A^\IN}$ and $\U\C_{A^\IN}^\infty\equivW\U\C_{A^\IN}$.
\end{enumerate}
\end{cor}

% Fix Indentation
\noindent
The example $A=\IN$ and Proposition~\ref{prop:LPO-lim} show that the first reductions cannot be strengthened to equivalences.

\begin{cor}
$\LLPO^\infty\equivSW\C_2^\infty\equivSW\C_{2^\IN}\equivSW\WKL$ and $\C_\IN^\infty\equivSW\lim\lW\C_{\IN^\IN}$.
\end{cor}

% Fix Indentation
\noindent
We obtain the following interesting fixed points of the infinite loop operation, 
using Proposition~\ref{prop:monotonicity} and the fact
that $\C_{2^\IN}\equivSW\WKL$ and $\C_{\IN^\IN}$ are cylinders.\footnote{The fact that $\WKL$ 
is closed under infinite loops was independently also proved by Pauly et al.\ (unpublished, personal communication).}

\begin{cor}
$\C_{2^\IN}^\infty\equivSW\C_{2^\IN}$, $\C_{\IN^\IN}^\infty\equivSW\C_{\IN^\IN}$, $\U\C_{\IN^\IN}^\infty\equivSW\U\C_{\IN^\IN}$ and $\WKL^\infty\equivSW\WKL$.
\end{cor}

% Fix Indentation
\noindent
Some classes of problems can be characterized as cones of certain problems, e.g., $\WKL$ characterizes
the problems that are typically called non-deterministically computable, $\C_\IN$ characterizes the problems
that are computable with finitely many mind changes and $\lim$ characterizes the cone of problems
that are limit computable (for single-valued problems on computable Polish spaces, $\C_{\IN^\IN}$ characterizes
the effectively Borel measurable problems). See~\cite{BBP12} for more details.
Hence, the corresponding classes are or are not closed under infinite loops.

\begin{cor}
We obtain:
\begin{enumerate} 
\item The class of computable problems, the class of non-deterministically computable problems
and the class of single-valued effectively Borel measurable functions (with suitable domains) are closed
under infinite loops.
\item The class of problems that are computable with finitely many mind changes and the class
of problems that are limit computable are not closed under infinite loops.
\end{enumerate}
\end{cor}

\section{Loops on computability-theoretic problems}
\label{sec:NON}

In this section we want to discuss the peculiar situation for problem $f:\In\DD\mto\DD$ on the set
$\DD$ of {\em Turing degrees}. We assume that $\DD$ is represented by $\deg:\IN^\IN\to\DD$,
where $\deg(p):=\{q\in\IN^\IN:p\equivT q\}$ denotes the Turing degree of $p\in\IN^\IN$.
The essential observation is that {\em while loops} are not more powerful than {\em for loops} for problems
on Turing degrees. The reason is that the stop condition that measures success of a while loop
cannot exploit any useful information on Turing degrees (see also the discussion of densely realized problems
in~\cite{BHK17a,BP18}).

\begin{prop}[Problems on Turing degrees]
\label{prop:Turing}
Let ${f:\In\DD\mto\DD}$ be a problem. Then we obtain 
\begin{enumerate}
\item $f^\diamond\equivSW f^{[*]}$,
\item $\widehat{f^\diamond}\equivSW f^\omega$, if $f$ is pointed.
\end{enumerate}
\end{prop}
\begin{proof}
(1) By Proposition~\ref{prop:for-while} we need to prove $f^\diamond\leqSW f^{[*]}$.
Given a name $p$ of an input $\deg(p)\in\dom(f^\diamond)$, there has to be a successful finite run of the loop on $f$
with some finite number $n\in\IN$ of loops starting from the name $p$. 
During these loops $f$ is applied to data that have been previously computed, i.e., there are
$p_0,p_1,...,p_n\in\IN^\IN$ with $p=p_0$ such that $\deg(p_{i+1})\in f(\deg(p_i))$ and there is a computable 
function $g:\In\IN^\IN\to\IN$
such that $g\langle p_0,...,p_n\rangle=0$ signals success of the finite run. As $g$ needs to be continuous,
one can systematically search for $n\in\IN$ and $w_0,...,w_n\in\IN^*$ such that
the prefixes $w_i\prefix p_i$ are already sufficient for $g$ to signal success.
Now the computation of $f^\diamond$ on the name $p$ can be simulated with 
the help of $f^{[n]}$ on input $p$, as the results $\deg(q_{i+1})\in f(\deg(q_i))$ for $i=0,...,n$ and $q_0=p$
also satisfy $\deg(w_{i+1}q_{i+1})\in f(\deg(q_i))$ and hence the run of $f^\diamond$ with these modified
results $w_{i+1}q_{i+1}$ will also signal success after $n$ loops. \\
(2) Together with Proposition~\ref{prop:omega} we obtain $\widehat{f^\diamond}\leqSW \widehat{f^\omega}\equivSW f^\omega\leqSW\widehat{f^\diamond}$ and hence $\widehat{f^\diamond}\equivSW f^\omega$ for pointed $f$.
\end{proof}

% Fix Indentation
\noindent
However, unlike in the case of single-valued problems, $f^\infty$ remains more powerful
than $f^\omega$ for problems on Turing degrees in general.
We use again the sequence $(p_n)_{n\in\IN}$ from Lemma~\ref{lem:Turing} and a function similar to the
one from Example~\ref{example:diamond-omega} to construct a counterexample.
By $a_n:=\deg(p_n)$ we denote the Turing degree of $p_n$ for all $n\in\IN$.

\begin{exa}
\label{example:infty-omega}
We consider $f:\In\DD\mto\DD$
with $\dom(f)=\{0\}\cup\{a_n:n\in\IN\}$ with
\[f(a):=\left\{\begin{array}{ll}
a_0 & \mbox{if $a=a_0$}\\
a_n & \mbox{if $a=a_{n+1}$}\\
\{a_n:n\in\IN\} & \mbox{if $a=0$}
\end{array}\right.\]
Then $f^\infty\nleqW f^\omega\equivSW\widehat{f^\diamond}$.
\end{exa}
\begin{proof}
We consider the problem $g:\DD\to\DD$ with $g(a)=a_0$. We claim that $g\leqW f^\infty$.
This is because the output of $f^\infty$ on any input can be used to produce an output of 
degree $a_0$.
This is because $f(a)$ yields some $a_n$ with $n\in\IN$ and $n$ further applications of $f$ yield $a_0$.
On the other hand, $f^\omega(0)$ has
a possible output of Turing degree equal to $a_1<a_0$, as $f^{[n]}(0)$ has a possible output
of Turing degree equal to $a_{1}$ for every $n\in\IN$. 
Thus, $f^\infty\nleqW f^\omega$. Finally, $f^\omega\equivSW\widehat{f^\diamond}$ holds by Proposition~\ref{prop:Turing}.
\end{proof}

% Fix Indentation
\noindent
It would be nice to have a more natural example of a problem for which $f^\infty$ and $\widehat{f^\diamond}$
are not equivalent. However, our results show that for many natural problems $f$ we actually obtain 
$f^\infty\equivSW\widehat{f^\diamond}$.

\begin{cor}
The problems $\LPO$, $\LLPO$, $\C_\IN$, $\C_{2^\IN}$, $\C_{\IN^\IN}$, and $\lim$
in the role of $f$ all satisfy $f^\infty\equivSW\widehat{f^\diamond}$.
 \end{cor}

% Fix Indentation
\noindent
Analogously, this also holds for many further problems.
A natural candidate for a problem that does not satisfy $f^\infty\equivSW\widehat{f^\diamond}$
is the {\em non-computability problem}, which was studied,
for instance, in~\cite{BHK17a,Bra23}:
\[\NON:\IN^\IN\mto\IN^\IN,p\mapsto\{q\in\IN^\IN:q\nleqT p\}.\]
This problem can also be seen as a problem on Turing degrees.
Hence, Proposition~\ref{prop:Turing} yields the following.

\begin{cor}
\label{cor:NON}
$\NON^\diamond\equivSW\NON^{[*]}$ and $\widehat{\NON^\diamond}\equivSW \NON^\omega$.
\end{cor}

% Fix Indentation
\noindent
It is easy to see that loops of this problem are related to the following auxiliary problems, which are of independent interest,
namely the problems of finding increasing and decreasing chains of Turing degrees above the input:

\begin{defi}[Increasing and decreasing chains of Turing degrees]
We consider the following problems:
\begin{enumerate}
\item $\TDINC:\IN^\IN\mto\IN^\IN,p\mapsto\{\langle p_0,p_1,p_2,...\rangle\in\IN^\IN:(\forall i)\;p\lT p_i\lT p_{i+1}\}$.
\item $\TDDEC:\IN^\IN\mto\IN^\IN,p\mapsto\{\langle p_0,p_1,p_2,...\rangle\in\IN^\IN:(\forall i)\;p\lT p_{i+1}\lT p_{i}\}$.
\end{enumerate}
\end{defi}

% Fix Indentation
\noindent
First, we prove that infinite loops of $\NON$ are equivalent to $\TDINC$.

\begin{prop}
\label{prop:TDINC}
$\NON^\infty\equivW\TDINC$.
\end{prop}
\begin{proof}
Firstly, it is clear that $\TDINC\leqW\NON^\infty$: given $p\in\IN^\IN$, we can use $\NON$
to compute some $q$ with $q\nleqT p$ and hence $p_0:=\langle q,p\rangle$ satisfies $p\lT p_0$.
Using $\NON$ on input $p_0$ we obtain $q_0$ with $q_0\nleqT p_0$ and hence $p_1:=\langle q_0,p_0\rangle$
satisfies $p_0\lT p_1$. If we continue inductively like this, we obtain an increasing chain $\langle p_0,p_1,p_2,...\rangle\in\TDINC(p)$.

For the inverse reduction $\NON^\infty\leqW\TDINC$ we consider some input $p$ for $\NON^\infty$
and compute $\langle p_0,p_1,p_2,...\rangle\in\TDINC(p)$. If in the infinite loop $\NON^\infty$ the
problem $\NON$ is applied to some input that has been computed from $p$, then $p_0$ is a legitimate answer.
If in the next loop $\NON$ is applied to some result that has been computed from $p_0$, then $p_1$ is a 
legitimate answer and so forth. Hence, the increasing chain $\langle p_0,p_1,p_2,...\rangle$ can be used
to compute answers to all applications of $\NON$ in the course of the infinite loop. 
\end{proof}

% Fix Indentation
\noindent
On the other hand, it is easy to see that $\NON^\omega$ can be computed from $\TDDEC$.

\begin{prop}
$\NON^\omega\leqW\TDDEC$.
\end{prop}
\begin{proof}
Given an input $p\in\IN^\IN$ to $\NON^\omega$ and a decreasing chain 
\[p\lT ...\lT p_2\lT p_1\lT p_0\]
of Turing degrees,
we can use $\langle p_{n-1},p_{n-2},...,p_0\rangle$ in order to determine an answer for $\NON^{[n]}(p)$ and in this way we can
compute a solution in $\NON^\omega(p)$.
\end{proof}

% Fix Indentation
\noindent
We conjecture that $\NON^\omega$ is strictly weaker than $\NON^\infty$.

\begin{conj}
\label{conjecture:NON-omega}
$\NON^\omega\lW\NON^\infty$.
\end{conj}

% Fix Indentation
\noindent
Given the results of this section, this conjecture is a consequence of the next
conjecture.

\begin{conj}
$\TDINC\nleqW\TDDEC$.
\end{conj}

% Fix Indentation
\noindent
One might be tempted to prove this separation by considering an initial segment of Turing
degrees similar as in Lemma~\ref{lem:Turing}. However, such an initial segment cannot
be uniformly computable in the maximal element by results in~\cite{Ish02}.
Hence, additional ideas are required to resolve the two conjectures stated here.\footnote{According to the author's understanding, Patrick Uftring has recently solved these conjectures in the affirmative (personal communication).}
However, one can use initial segments in order to prove the following result.

\begin{prop}
$\NON^{[n+1]}\nleqW\NON^{[n]}$ for all $n\in\IN$.
\end{prop}
\begin{proof}
Similar as in Proposition~\ref{prop:TDINC} one can prove that the problem $f_n:\DD\mto\DD^n$ with
\[f_n(a):=\{(a_1,...,a_n)\in\DD^n:a<a_1<a_2<...<a_n\}\]
satisfies $\NON^{[n]}\equivW f_n$ for all $n\in\IN$.
Now, upon input of $0$, the problem $f_n$ can yield an output $(a_1,...,a_n)$
for an initial segment
\[0<a_1<a_2<...<a_n\]
of the Turing degrees. From $(a_1,...,a_n)$ no other Turing degrees outside of this
initial segment are computable, in particular, no chain of Turing degrees of length $n+1$
above $0$ can be computed.
\end{proof}

\bibliographystyle{alphaurl}
\bibliography{limits}

\section*{Acknowledgments}
The author was funded by the German Research Foundation (DFG, Deutsche Forschungsgemeinschaft) -- project number 554999067 and by the National Research Foundation of South Africa (NRF) -- grant number 151597.
We would like to thank Arno Pauly, Manlio Valenti and Hendrik Smischliaew for helpful discussions.

\end{document}